\documentclass{amsart}
\usepackage[utf8]{inputenc}  
\usepackage[T2A]{fontenc}     
\usepackage [english]{babel}

\usepackage{amscd}
\usepackage{amsfonts}
\usepackage{amsmath}
\usepackage{amssymb}
\usepackage{amsthm}
\usepackage{geometry}
\usepackage{indentfirst}
\usepackage{mathrsfs}
\usepackage{tikz-cd}
\usepackage[matrix,arrow,curve]{xy}
\usepackage[]{youngtab}

\usepackage[T2A]{fontenc}
\usepackage[english,]{babel}
\usepackage{graphicx}
\usepackage{amsfonts}
\usepackage{amsmath}
\usepackage{amssymb}
\usepackage{amscd}

\usepackage{hyperref}
\usepackage{color}

\theoremstyle{plain}
\newtheorem{prop}{Proposition}[section]
\newtheorem{theo}[prop]{Theorem}
\newtheorem{coro}[prop]{Corollary}

\theoremstyle{definition}
\newtheorem{defi}[prop]{Definition}

\newtheorem{remar}[prop]{Remark}

\newcommand{\lra}{\longrightarrow}
\newcommand{\ra}{\rightarrow}

\newcommand{\cC}{{\mathcal C}}

\newcommand{\cE}{{\mathcal E}} 
\newcommand{\cF}{{\mathcal F}}

\newcommand{\cK}{{\mathcal K}}

\newcommand{\cW}{{\mathcal W}}

\newcommand{\bC}{\mathbf{C}}

\newcommand{\bI}{\mathbf{I}}

\newcommand{\bK}{\mathbf{K}}

\newcommand{\bbC}{\mathbb{C}}

\newcommand{\bbQ}{\mathbb{Q}}

\newcommand{\bbS}{\mathbb{S}}

\newcommand{\bbZ}{\mathbb{Z}}

\newcommand{\ul}[1]{\underline{#1}}

\newcommand{\fU}{{\mathfrak U}}
\newcommand{\fX}{{\mathfrak X}}
\newcommand{\fY}{{\mathcal Y }}

\newcommand{\rH}{{\mathrm H}}
\newcommand{\rK}{{\mathrm K}}

\newcommand{\AffC}{{\mathrm{Aff}_\bC^\mathrm{f.t.}}}

\newcommand{\AffCsm}{{\mathrm{Aff}_\bC^\mathrm{sm}}}

\newcommand{\an}{\mathrm{an}}

\newcommand{\rB}{\mathrm{B}}

\newcommand{\Ch}{\mathrm{Ch}}
\newcommand{\coker}{\mathrm{coker}}

\newcommand{\cl}{\mathrm{cl}}

\newcommand{\DGCAT}{\mathrm{DGCAT}}

\newcommand{\Fun}{\mathrm{Fun}}

\newcommand{\han}{\mathrm{han}}
\newcommand{\HC}{\mathrm{HC}}
\newcommand{\HH}{\mathrm{HH}}

\newcommand{\Hom}{\mathrm{Hom}}
\newcommand{\HP}{\mathrm{HP}}
\newcommand{\bHP}{\mathbf{HP}}

\newcommand{\Lan}{{\mathrm{Lan}}}
\newcommand{\Loc}{\mathrm{Loc}}

\newcommand{\LOC}{\mathrm{LOC}}

\newcommand{\op}{{op}}
\newcommand{\Or}{\mathrm{Or}}
\newcommand{\Perf}{\mathrm{Perf}}

\newcommand{\red}{\mathrm{red}}

\newcommand{\SchC}{{\mathrm{Sch}_\bC^\mathrm{f.t.}}}

\newcommand{\SchCsm}{{\mathrm{Sch}_\bC^\mathrm{sm}}}

\newcommand{\Shv}{\mathrm{Shv}}
\newcommand{\Sing}{\mathrm{Sing}}

\newcommand{\sm}{\mathrm{sm}}
\newcommand{\Sp}{\mathrm{Sp}}
\newcommand{\Spc}{\mathrm{Spc}}
\newcommand{\Spec}{\mathrm{Spec}}
\newcommand{\st}{\mathrm{st}}

\newcommand{\TC}{\mathrm{TC}}
\newcommand{\THH}{\mathrm{THH}}
\newcommand{\tp}{\mathrm{top}}
\newcommand{\rTop}{\mathrm{Top}}

\DeclareMathOperator*{\colim}{colim}

\author{Andrei Konovalov}

\title{Nilpotent invariance of semi-topological K-theory of dg-algebras and the lattice conjecture}
\thanks{The author was partially supported by Basis Foundation grant 18-1-6-95-1, Leader
(Math), by the HSE University Basic Research Program, and by Simons-IUM fellowship}
\address{National Research University Higher School of Economics, Russian Federation}
\email{kon\_an\_litsey@list.ru, akonovalov@hse.ru}

\begin{document}

\maketitle

\begin{abstract}
    We show existence of a natural rational structure on periodic cyclic homology, conjectured by L. Katzarkov, M. Kontsevich, T. Pantev, for several classes of dg-categories, including proper connective $\bbC$-dg-algebras and dg-categories of local systems. The main ingredient is derived nilpotent invariance of A. Blanc's semi-topological K-theory, which we establish along the way.
    
\end{abstract}

\section{Introduction}

One of the features of algebraic geometry over the field of complex numbers is existence of a pure Hodge structure on de Rham cohomology of smooth proper varieties. 
Noncommutative algebraic geometry studies $k$-dg-categories, aka noncommutative schemes. A fruitful idea is 
in the realm of dg-categories one can still define many properties and invariants of schemes, such as smoothness, properness, (direct sums of) 
Dolbeault cohomology groups and de Rham cohomology, --  the correspondence between the commutative and noncommutative worlds being the functor assigning to a scheme $X/k$ the dg-category of perfect complexes 
$\Perf(X)$. The last two invariants are presented by Hochschild homology and periodic cyclic homology and they come with the Hodge-de Rham spectral sequence, which degenerates for smooth proper dg-categories when $k$ is a field of characteristic 0 (\cite{Kal}
). 
In \cite{KKP}, the authors consider noncommutative schemes over $\bbC$ and suggest to look for a counterpart for Hodge structures, which are already partially presented thanks to the aforementioned degeneration of the spectral sequence. One of the missing parts is a natural rational structure, i.e. a functor $F: dgCat_\bbC \lra Mod_\bbQ$ and a natural transformation $F \ra \HP(\cdot/\bbC)$, such that, for a smooth proper dg-category $T$, the induced morphism $F(T) \otimes_\bbQ\bbC \ra \HP(T/\bbC)$ is an equivalence (i.e. a quasi-isomorphism).

In this paper, we consider the topological K-theory of dg-categories functor $\rK^\tp$, defined by A. Blanc as a promising candidate for the role of integral structure. We prove a statement about rational structure, which we later will refer to as the "lattice conjecture", in several cases, which are listed below (cf. Corollary~\ref{lcnil}, Theorem~\ref{lccnprop}, Corollary~\ref{lcdsch}, Theorem~\ref{lclocsys}).

\begin{theo}\label{lcall}

Let $LC\subset dgCat_\bbC$ be the full subcategory of dg-categories on which the natural transformation of functors $\rK^\tp(T)\otimes \bbC \ra \HP(T/\bbC)$ is an equivalence. Then $LC$ contains the following classes of dg-categories: a) $T = \Perf(B)$ where $B$ is a connected proper dg-algebra; b) $T = \Perf(B)$ where $B$ is a connected dg-algebra, such that $\rH_0B$ is a nilpotent extension of a commutative $\bbC$-algebra of finite type; c) $T = \Loc(M, \mathbb{C})$ where $M$ is a connected locally contractible space with some condition on its fundamental group (see Theorem~\ref{lclocsys});
d) $T = \Perf(\fX)$ where $\fX$ is a derived $\bbC$-scheme, such that its classical part is a separated scheme of finite type.
$LC$ satisfies 2-out-of-3 property with respect to exact triples of dg-categories and is closed under Morita-equivalences and taking retracts.

\end{theo}

In section 4.7, \cite{Bla}, the author considered finite-dimensional classical algebras and used a variant of $\rK^\tp$, called pseudo-connective topological K-theory, to provide periodic cyclic homology of such algebras with a rational structure. Since these algebras lie in the class a) of Theorem~\ref{lcall}, it follows that better-behaving topological K-theory works just as well, which can also be seen directly (see Proposition~\ref{Kstneg}). We also explain in Proposition~\ref{lcfddg} that Orlov's result (Theorem 2.19, \cite{Orl}) implies that finite-dimensional smooth $\bbC$-dg-algebras also lie in $LC$.

Topological K-theory of dg-categories is defined using a topological realization functor, which generalizes the procedure of analytification from $\bbC$-varieties to arbitrary (spectrum-valued) invariants of schemes. To prove  Theorem~\ref{lcall}, we study the behaviour of the realization functor, focusing on the case of Hochschild-type invariants. It allows us to establish the following result, crucial for proving the main theorem.

\begin{theo}\label{Kstnilintro}

Let $v: B \ra A$ be a nilpotent extension of connective $\bbC$-dg-algebras. Then the induced map $\rK^\st(B) \ra \rK^\st(A)$ is an equivalence.

\end{theo}

After derived nil-invariance is established, proving most of the cases does not require much work, but for the case of local systems on $M$ we need to understand topological K-theory of group algebras. This we can do only under some assumptions on the group, which corresponds to putting assumptions on the fundamental group of $M$. Concretely, we ask the group to
satisfy the Burghelea conjecture and the rational Farrell-Jones conjecture, which are both established for a large class of groups, -- and under these assumptions we prove the lattice conjecture. We also suggest a new approach to constructing a counterexample to the Farrell-Jones conjecture.  

\ \ 

\subsubsection*{Structure of the paper}

In the first section, we state the two main theorems and recall some motivation behind the lattice conjecture. We also fix some conventions that will be used throughout the paper and sketch the structure of the exposition.

The second section is devoted to considering two realizations functors, which allow one to extend the functor of taking the space of complex points with analytic topology from schemes to spectral presheaves. We show that these two functors coincide, which allows us later to use properties of both.

The realization formalism was used in \cite{Bla} to define semi-topological K-theory, which after inverting the Bott element becomes topological K-theory. In the third section, we recall the necessary definitions and statements from \cite{Bla}. 

Semi-topological K-theory of a dg-category $T$ is built from algebraic K-theories of different base-changes of $T$. And, while algebraic K-theory itself is a very complicated invariant, to some extent, it can be approximated by Hochshild and (variants of) cyclic homology. In the fourth section, we consider realizations of Hochschild-type invariants. In particular, we show that the realizations of $\HH$ and $\HC$ vanish.

In the fifth section, we recall the definition of derived nilpotent invariance 
and prove the Theorem~\ref{Kstnilintro} using the computations from the previous section.

The last section is devoted to considering consequences of Theorem~\ref{Kstnilintro}. In particular, we prove Theorem~\ref{lcall} and sketch some other possible applications of our ideas.

\subsubsection*{Conventions}

Most of the time, we work over the field of complex numbers $\bbC$. This is because the very first constructions are using existence of an analytification functor, assigning to a separated scheme of finite type over $\bbC$ its space of complex points with analytic topology.

We freely use the language of $\infty$-categories as in the works of J. Lurie. We do not use heavy machinery though; the main benefit is to avoid a lot of cofibrant replacements when working with localizations and to work conveniently with categories enriched in homotopy types and spectra. 

By a left Kan extension we mean the $\infty$-categorical version, see \cite{LurHTT}, Definition 4.3.3.2, and similarly other colimits in $\infty$-categories (which may be presented as homotopy colimits in appropriate model structures). 

By $Mod_A$, we denote the dg-category of (right) modules over a dg-algebra $A$. By $\Perf(A)$ we denote the full subcategory of compact objects in $Mod_A$.

The dg-category of $A$-modules has a forgetful (aka generalised Eilenberg-MacLane) functor to the stable $\infty$-category $\Sp$ of spectra. Since we are working with K-theory, it would be convenient often to consider such invariants as $\HH(\cdot/k)$, $\HP(\cdot/k)$ as spectra forgetting the $k$-linear structure. The (semi-)topological K-theory functor will as well take values in the category of spectra, though, for the purposes of the lattice conjecture, it can be rationalized at any moment.

Small triangulated (=idempotent-complete pretriangulated) $\bbC$-dg-categories form an $\infty$-category $dgCat_\bbC$. To any small dg-category $T$ one can assign its triangulated envelope $\Perf(T)$, -- this can be thought of as the fibrant replacement functor in the Tabuada's model category of small $\bbC$-dg-categories and Morita-equivalences (\cite{Tab}). This theory is equivalent to the $\infty$-category of idempotent-complete $\bbC$-linear small stable $\infty$-categories (Corollary 5.5, \cite{Cohn}).

By a weakly localizing invariant, we mean a (Morita-invariant) functor 

$$
    F: dgCat_\bbC \lra \cC
$$
to a stable $\infty$-category $\cC$, which 
sends exact triangles of triangulated dg-categories to fiber sequences in $\cC$. Note that some authors call such functors "localizing"; in this paper, this term will be reserved for functors that also commute with filtered colimits ($\HP$ and $\HC^-$ are not localizing in this strong sense).

\subsubsection*{Acknowledgements}

I would like to thank Chris Brav, my advisor, for sharing his ideas and a lot of helpful discussions. I also would like to thank D. Kaledin and A. Prikhodko for many fruitful conversations we had while I was working on this paper, and S. Gorchinskiy for his attention to my Master's thesis where a special case of the main result was proved.


\ \ 

\section{Realization functors}

In this section, we introduce two ways of assigning the so-called topological realization spectrum to a spectral presheaf on the category of affine schemes of finite type over $\bbC$ and prove that, in fact, these two procedures produce the same spectrum. The realization formalism will later be used to define semi-topological K-theory of dg-categories following A. Blanc \cite{Bla}. Each of the two procedures has its own advantages, which will be used throughout the paper.

The results of this section are not new, they appear in a similar form in \cite{AH} with spaces instead of spectra as a target category. Besides introducing notation, this chapter is intended to clarify the proof of the comparison theorem (see also Remark~\ref{stupidKan}).

\ \

We denote by $\AffC\subset\SchC$ the category of affine schemes of finite type and the category of separated schemes of finite type over the field of complex numbers and by $\AffCsm\subset\SchCsm$ the corresponding full subcategories of smooth schemes. There is a standard functorial way to define a topology on the set of complex points of a separated scheme $X$ of finite type over $\bbC$. We denote by $X^\an$ this space and by $X^\han$ the corresponding homotopy type. So we have the functors

$$
\xymatrix{
\SchC \ar[r]^{(\cdot)^\han} \ar[d]^{(\cdot)^\an}& \Spc  \ar[r]^{\Sigma^\infty(\cdot)_+}& \Sp \\
\rTop_\cl ,
} 
$$
where we denote by $\Spc$ the $\infty$-category of homotopy types, by $\Sp$ the $\infty$-category of spectra and by $\rTop_\cl$ the 1-category of 
Hausdorff compactly generated topological spaces, where we use the subscript $_\cl$ to emphasize that the functor $(\cdot)^\an$ takes values in honest topological spaces instead of homotopy types. Note that this functor commutes with gluing of schemes, so $(\cdot)^\an$ is determined by its restriction to the subcategory of affine schemes. By abuse of notation, we will use $(\cdot)^\an$ and $(\cdot)^\han$ to denote the corresponding restrictions.

\ \ 

Now we want to extend the functor $(\cdot)^\han$
to simplicial 
presheaves on $\AffC$. The first, straightforward way to do this is via left Kan extension along the Yoneda embedding:

$$
\xymatrix{
\AffC \ar[r]^{(\cdot)^\han} \ar[d]^{Y}& \Spc  \ar[r]^{\Sigma^\infty(\cdot)_+}& \Sp \\
\Pr_{\Spc}(\AffC) \ar[ur]_{
Y^*(\cdot)^\han} , \\
} 
$$

The value of $Y^*(\cdot)^\han$ on a presheaf F can be expressed via the usual colimit formula for left Kan extensions:

\[
    Y^*(F)^\han \simeq \colim_{Y(X) \ra F \in Y/F} X^\han .
\]

\ \ 

The functors $Y^*(\cdot)^\han$ and $\Sigma^\infty(\cdot)_+$ commute with colimits, so by the universal property of stabilization (\cite{LurHA}, Corollary 1.4.4.5), we define $Re$:

$$
\xymatrix{
\AffC \ar[r]^{(\cdot)^\han} \ar[d]_{Y}& \Spc  \ar[r]^{\Sigma^\infty(\cdot)_+}& \Sp \\
\Pr_{\Spc}(\AffC) \ar[ur]_{Y^*(\cdot)^\han} \ar[d]   \\
\Pr_{\Sp}{\AffC} \ar[uurr]_{Re} , 
} 
$$
    using that the stabilization of $\Pr_{\Spc}(\AffC)$ is the stable $\infty$-category $\Sp(\Pr_{\Spc}(\AffC)) \simeq \Pr_{\Sp}{\AffC}$ (\cite{LurHA}, Remark 1.4.2.9). The functor $Re$ 
    has a right adjoint given by $E \mapsto \{ X \mapsto 
    \ul{\Hom}(\Sigma^\infty_+X^\han, E)   \}$.

\begin{remar} \label{stupidKan}

If instead of doing this two-step procedure, we take simply the left Kan extension of $\Sigma^\infty(\cdot)^\han_+$ to $\Pr_{\Sp}{\AffC}$, the resulting functor will not commute with the loops functor and its value on a presheaf $F$ will actually coincide with its value on the connective cover $\tilde F$ of $F$:

$$
    \Lan_{\Sigma^\infty(Y)_+} (\Sigma^\infty(\cdot)^\han_+) (F) \simeq \colim_{Y(X) \ra F} (\Sigma^\infty(X)^\han_+) \simeq \colim_{Y(X) \ra \tilde F} (\Sigma^\infty(X)^\han_+) \simeq \Lan_{\Sigma^\infty(Y)_+} (\Sigma^\infty(\cdot)^\han_+) (\tilde F) .
$$

\end{remar}

\ \ 

\begin{remar}

Note that we could do the same procedure, starting from the category $\AffC^\sm$. We denote the resulting functor by $Re^\sm$.

\end{remar}

Now we consider the other way to extend the analytification functor to $\Pr_{\Sp}{\AffC}$.

Take some $F \in \Pr_{\Sp}{\AffC}$. We can construct its left Kan extension along the complex points functor: 

$$
\xymatrix{
\AffC^\op \ar[r]^{F} \ar[d]^{(\cdot)^\an}& \Sp   \\
\rTop_\cl^\op \ar[ur]_{\an^*F} ,
} 
$$
and then take its value on the standard cosimplicial object $\Delta^\bullet_\tp: \Delta \lra \rTop_\cl$. Now we define $Re' F$ to be the realization of the resulting simplicial spectrum: $Re' F : = |(\an^*F)(\Delta^\bullet_\tp)|$. So we obtain a functor $Re'$ as the following composition:

 \[   Re' : \Pr{}\!_{\Sp}(\AffC) \xrightarrow{an^*} \Pr\!_{\Sp}{\rTop_\cl} \xrightarrow{(\Delta^\bullet_\tp)^*} \Pr\!_{\Sp}(\Delta) \xrightarrow{|\cdot|} \Sp . \] 

This construction was considered in \cite{FW01} 
and was used to define semi-topological K-theory for quasi-projective varieties (see also \cite{FW05}).

\ \ 

Now we prove a version of Antieau-Heller Theorem 2.3 \cite{AH}:

\begin{theo}\label{Recomparison}

    There is a natural equivalence between the functors $Re$ and $Re'$.

\end{theo}

\begin{proof}

To prove that $Re \simeq Re'$, note 
that both functors commute with colimits and shifts, hence it is sufficient to provide an equivalence on their restrictions to the subcategory of representable presheaves.

For $X\in \AffC$ we want to compute $Re'( \AffC (\cdot, X) )$. 
Since $\Sigma^\infty_+$ preserves left Kan extensions, by applying Yoneda lemma twice:  

\[  \forall H\in \Pr{}\!_{Set}(\rTop_{cl}) \ \ \Pr{}\!_{Set}(\AffC) (\AffC (\cdot, X) , H(\cdot)^\an) \simeq H(X^\an) \simeq \Pr{}\!_{Set}(\rTop_{cl})(\rTop_{cl}(\cdot, X^\an), H)
    ,
\]

and using the adjunction between left Kan extension and precomposition, we deduce

$$
    (an^*( \AffC (\cdot, X) ))(M) \simeq \Sigma^\infty( \rTop_{cl}(M, X^\an) )_+ .
$$

and, by functoriality, we get

$Re'( \AffC (\cdot, X) ) \simeq |[n] \mapsto \Sigma^\infty( Sing^n(X^\an) )_+ | \simeq \Sigma^\infty(X^\han)_+ \simeq Re(\AffC (\cdot, X) )$ and $Re'\simeq Re$.

\end{proof}

Unravelling the definitions from the next section
, we get the following corollary.

\begin{coro} \label{Kstformula}

There are natural isomorphisms:

$$
    \rK^\st (T) \simeq | [n] \mapsto  \colim_{\Delta^n_\tp \ra (\Spec R)^\an \in (\cdot)^\an/\Delta^n_\tp}\rK(T\otimes_\bbC R)| ,
$$

$$
        \tilde\rK^\st(T) \simeq | [n] \mapsto  \colim_{\Delta^n_\tp \ra (\Spec R)^\an \in (\cdot)^\an/\Delta^n_\tp}\tilde\rK(T\otimes_\bbC R)| .
$$

\end{coro}

These formulae have the advantage as being colimits \it of \rm K-theory spectra instead of having colimits 
 \it over \rm some K-theory functor.
 
 \ \ 

\section{K-theories of dg-categories and topological Chern character}

In this section we give a brief survey of results in \cite{Bla}.

\ \ 

Consider the functors

$$
    \rK(T\otimes_\bbC \cdot ), \  \tilde\rK(T\otimes_\bbC \cdot): \AffC^\op \lra \Sp ,
$$
assigning to an affine scheme $S = \Spec R$ the spectra of algebraic K-theory and connective algebraic K-theory of the dg-category $T\otimes_\bbC R$ respectively. Note, that, due to Morita-invariance of K-theory, it essentially does not matter if one replaces the dg-category $T\otimes_\bbC R$ by $T\otimes_\bbC \Perf(S)$.

Using topological realization, one defines semi-topological K-theory:

\begin{defi} (\cite{Bla}, Definition 4.1)
The nonconnective/connective \it semi-topological K-theory \rm of a $\bbC$-dg-category $T$ is the spectrum $\rK^\st(T):= Re(\rK(T\otimes_\bbC\cdot))$ / $\tilde\rK^\st(T):= Re(\tilde\rK(T\otimes_\bbC\cdot))$. By monoidality of Re (\cite{Bla}, Proposition 3.12
), 
$\rK^\st$ can be considered here as a functor from $\rK(\cdot)$-modules to $\rK^\st(\bbC)$-modules and analogously for the connective version.

\end{defi}

The coming 
definition of topological K-theory is based on the following calculation.

\begin{theo} \label{Kstpoint} (\cite{Bla}, Theorems 4.5, 4.6)
There are canonical equivalences of ring spectra $\rK^\st(\bbC) \simeq \tilde\rK^\st(\bbC) \simeq ku$, where $ku$ is the connective cover of the topological K-theory spectrum $KU$.

\end{theo}

Here the second equivalence is due to a direct computation and the first one comes from vanishing of negative algebraic K-theory for regular affine schemes and the following theorem:

\begin{theo} (\cite{FW03}, Corollary 2.7 or \cite{Bla}, Theorem 3.18)  
For any spectral presheaf $F\in \Pr{}\!_{\Sp}(\AffC)$, the functor

$$
    l^*: \Pr{}\!_{\Sp}(\AffC) \lra \Pr{}\!_{\Sp}(\AffC^\sm) ,
$$

which restricts a given presheaf to the subcategory of affine smooth schemes, produces an equivalence: $Re^\sm(l^*F) \simeq Re(F)$.

\end{theo}

\begin{remar}

Due to Theorem~\ref{Recomparison}, for $X$ a projective weakly normal variety, $\rK^\st(X)$ is equivalent to $\cK^{semi}(X)$  
of Friedlander-Walker, which was defined in terms of analytification of the ind-scheme  
$Mor(X, Gr)$. In particular, if $X = \Spec\bbC$, $\cK^{semi}(\bbC)$ is the connective spectrum corresponding to the group-completion of \[ Gr^\an = \bigsqcup_{n\ge 0}Gr(n, \infty)^\an, \] which is $BU\times \bbZ$, so the statement of Theorem~\ref{Kstpoint} (\cite{Bla}, Theorems 4.5) follows. And similarly for the statement of Proposition~\ref{lcsch} (\cite{Bla}, Theorems 4.5) below.

\end{remar}

\ \ 

Note that $KU \simeq ku[\beta^{-1}] \simeq \rK^\st(\bbC)[\beta^{-1}]$ for an appropriate choice of the Bott generator $\beta$ in $\pi_2(\rK^\st(\bbC))$. Now define topological K-theory

\begin{defi} (\cite{Bla}, Definition 4.13)
The \it topological K-theory \rm of a $\bbC$-dg-category $T$ is the $KU$-module $\rK^\st(T)[\beta^{-1}]$.

\end{defi}

This definition is compatible with the classical one:

\begin{prop} \label{lcsch} (\cite{Bla}, Proposition 4.32)
Let $X$ be a separated $\bbC$-scheme of finite type. Then there exists a canonical isomorphism: $\rK^\tp(Perf(X)) \simeq KU(X^\han)$.

\end{prop}

\ \ 

So, at least in the case of decent schemes, (semi-)topological K-theory, unlike algebraic K-theory, tends to have not very large 
homotopy groups. There is a topological Chern character {$\Ch^\tp:~\rK^\tp \ra \HP$}, 
and $\rK^\tp(T)\otimes \bbQ$ can be considered as a candidate for the role of a rational structure in the periodic cyclic homology $\HP(T)$. 

\ \ 

By establishing compatibility with the classical topological Chern character $\Ch_{utop}: KU(X^\han) \ra \rH\bbC[u^{\pm 1}](X^\han)$, Blanc proves the lattice conjecture in the case of schemes (\cite{Bla}, Proposition 4.32).

Blanc also considered the simplest noncommutative case, the case of finite-dimensional dg-algebras, in which the lattice conjecture was proved after replacing $\rK^\tp$ by $\tilde\rK^\tp$, which in general has worse properties (in particular, there is no simple reason for $\tilde\rK^\tp$ to be a localizing invariant). In our Theorem~\ref{lccnprop} and Proposition~\ref{lcfddg}, we consider two different generalizations of the aforementioned case and prove the lattice conjecture for each of them. In particular, it follows that there is no need for replacing $\rK^\tp$ by $\tilde\rK^\tp$ in \cite{Bla}, Proposition 4.37, either by dimension comparison or by 
our Proposition~\ref{Kstneg}, whcih directly shows that the natural map $\tilde\rK^\tp \ra \rK^\tp$ is an equivalence in this case.

\ \ 

\section{Topological realization of Hochschild-type invariants}

In this section, we consider behaviour of the topological realization applied to Hochschild-type invariants. The main statement here is that the topological realization of negative cyclic homology coincides with the topological realization of periodic cyclic homology, thus is nil-invariant in the sense that we will make precise in the next section.

\ \ 

From now on, $k$ is any subfield of $\bbC$, so one can consider $\bbC$-dg-categories as $k$-dg-categories and define Hochschild-type invariants in the $k$-linear setting. 
By forgetting the $H\mathbb{Z}$-module structure, we consider these invariants as functors to the category of spectra.

We start by 
computing the topological realization of the relative version of Hochschild homology.

\begin{prop}\label{ReHH}

Let $T$ be a 
$\bbC$-dg-category. Then the topological realization of the functor

$$
    \HH(T\otimes_\bbC \cdot /k): \AffC^\op \lra \Sp
$$

$$
    R \mapsto \HH(T\otimes_\bbC R /k ) 
$$

is trivial.

\end{prop}

\begin{proof}

We start with noticing that, since  
Hochschild homology is a monoidal functor (this seems to be a folklore statement, see Proposition 2.1 of \cite{AV} for a proof and §4.2.3 of \cite{Lod} for an earlier incarnation of this phenomenon) and topological realization is monoidal (by Proposition 3.12 \cite{Bla}), 
$Re\HH(\cdot/k)$ is a connective ring spectrum and $Re(\HH(T\otimes_\bbC \cdot/k))$ is a $Re\HH(\cdot/k)$-module.
So it suffices to prove that $Re\HH(\cdot/k)$ is trivial.
We show that its $\pi_0$-group is trivial, hence the identity morphisms on $Re\HH(\cdot/k)$ and on $Re(\HH(T\otimes_\bbC \cdot/k))$ are trivial as well.

By 
Proposition 3.10 \cite{Bla}, it suffices to prove that every two elements $a,b\in \HH_0(\bbC/k) \simeq \bbC$ are algebraically equivalent, i.e., for any two morphisms $h_\bbC \rightrightarrows \HH(\cdot/k)$, there exists an algebraic curve $C$ and a commutative diagram

$$
\xymatrix{
h_\bbC \ar[dr]^{a} \ar[d] \\
h_C \ar[r] & \HH(\cdot/k)   \\
h_\bbC \ar[ur]_{b} \ar[u] . 
} 
$$

Take $C = \Spec \bbC[t]$, the morphisms $\Spec \bbC \ra C$ given by the ideals $(t-a)$ and $(t-b)$ and the morphism $h_C \ra \HH(\cdot/k)$ given by $t$. The diagram is commutative by definition, so we are done.

\end{proof}

Now we can deduce the following important statement.

\begin{prop} \label{ReHC}

Let $T$ be a (smooth proper) $\bbC$-dg-category. Then

a) the topological realization of the functor

$$
    \HC(T\otimes_\bbC \cdot/k) : \AffC^\op \lra \Sp
$$

is zero;

b) the topological realization of the morphism of the functors

$$
    \HC^{-}(T\otimes_\bbC \cdot/k) \lra \HP(T\otimes_\bbC \cdot/k): \AffC^\op \lra \Sp
$$

is an
equivalence.

\end{prop}

Here by $\HC$, $\HC^-$ and $\HP$ we denoted the cyclic homology, the negative cyclic homology and the periodic cyclic homology respectively. 

\begin{proof}

The cofiber of the map $\HC^-(T\otimes_\bbC \cdot/k) \lra \HP(T\otimes_\bbC \cdot/k)$ is $\HC(T\otimes_\bbC \cdot/k)[2]$, so the claim b) follows from a), since topological realization commutes with colimits and shifts.

To prove a), we proceed by 

$$
    Re(\HC(T\otimes_\bbC \cdot/k)) \simeq Re (\colim_{BS^1}\HH(T\otimes_\bbC \cdot/k)) \simeq \colim_{BS^1}Re\HH(T\otimes_\bbC \cdot/k) \simeq 0,
$$

using commutation with colimits and the previous proposition.

\end{proof}

Since the realization of negative cyclic homology coincides with the realization of periodic cyclic homology, it, in a sense, accumulates good properties of both theories. This idea will be used in the next section.

\ \ 

\section{Nilpotent extensions of connective dg-algebras}

In this section, we remind the definition of (derived) nil-invariants and 
theorems about nil-invariance of the fiber of the cyclotomic trace $\rK^{inv}$ and of periodic cyclic homology. 
 We use these theorems to prove nil-invariance of semi-topological K-theory.

\ \ 

Suppose that $v: B \ra A$ is a map of connective dg-algebras or connective ring spectra, such that the induced map $\pi_0v$ is a surjection with a nilpotent kernel $I$:  $\pi_0 A \simeq \pi_0 B/I$. Then we call the map $v$ a \it nilpotent extension \rm of 
connective ring spectra. We say that the functor from the category of (connective) dg-algebras is \it nil-invariant \rm if it induces equivalences on nilpotent extensions. The following two theorems 
provide two important examples of nil-invariants: 

\begin{theo}\label{GHP} (\cite{G85})
Let $k$ be a field of characteristic 0 and $v: B \ra A$ be a nilpotent extension of connective $k$-dg-algebras. Then the induced map $\HP(B/k) \ra \HP(A/k)$ is an equivalence. 

\end{theo}

\begin{theo}\label{DGMTC} (\cite{DGM}, see also \cite{Ras} for a modern presentation)
Let $v: B \ra A$ be a nilpotent extension of connective ring spectra. Then the induced map 
\[\mathrm{fib}(\tilde\rK(B) \xrightarrow{tr(B)} \TC(B)) \ra \mathrm{fib} (\tilde\rK(A) \xrightarrow{tr(A)} \TC(A))\] 

is an equivalence.

\end{theo}

Here $tr: \tilde\rK \ra \TC$ is the cyclotomic trace map 

\begin{remar}

Using the stable structure, the last theorem can be reformulated as a statement about equivalence of $\coker(\tilde\rK(B) \ra \tilde\rK(A))$ and $\coker(\TC(B) \ra \TC(A))$. The former cokernel is equivalent to $\coker(\rK(B) \ra \rK(A)$) due to nil-invariance of non-positive K-theory.

\end{remar}

\begin{prop}\label{Kneg}

Let $v: B \ra A$ to be a nilpotent extension of connective $\bbS$-algebras. Then the induced map on the coconnective truncations of algebraic K-theory $\tau_{\le 0}\rK(B) \ra \tau_{\le 0}\rK(A)$ is an equivalence. 

\end{prop}

This seems to be a folklore statement. We learned it in this generality from S. Raskin; the experts probably knew it for a long time.

\begin{proof}[Sketch of the proof]

We proceed in two steps.

    The first step is to show nil-invariance of $\rK_0$. For a connective ring $R$, $\rK_{0,1}(R) \simeq \rK_{0,1}(\pi_0R)$ (\cite{DGM}, 3.1.2 or \cite{LurKM}, Lecture 20, Corollaries 3, 4), so we can reduce to classical rings and then to Noetherian classical rings using that K-theory commutes with filtered colimits. For a classical Noetherian nilpotent ring extension $A = B/I$, the natural correspondence
    
$$
    Proj^{fg}(B) \lra  Proj^{fg}(A)
$$
$$
    P \mapsto P' \simeq P\otimes_BA
$$

is a bijection on the sets of isomorphism classes, which can be shown by reduction to the square-zero case and application of an explicit process of lifting idempotents.

The second step is to use Bass's formula 
to deduce nil-invariance for $\rK_{i-1}$ from nil-invariance for $\rK_i$ and to proceed by induction on $-i$.

\end{proof}

\ \ 

Now we use the theorems above and the statements from the previous sections to prove the main result:

\begin{theo}\label{Kstnil}

Let $v: B \ra A$ be a nilpotent extension of connective $\bbC$-dg-algebras. Then the induced map $\rK^\st(B) \ra \rK^\st(A)$ is an equivalence.

\end{theo}

\begin{proof}

We want to use the $Re'$ presentation of semi-topological K-theory to express it as a colimit of algebraic K-theories, then to apply Theorem~\ref{DGMTC} termwise, and then, after commuting some colimits, use our comparison of the realizations of negative cyclic homology and periodic homology, together with nil-invariance of the latter. 

\  \ 

Remind that, by Corollary~\ref{Kstformula}, the map $\rK^\st(B)\ra \rK^\st(A)$ can be expressed as

$$
    | \colim_{\Delta^n_\tp \ra (\Spec R)^\an}\rK(B\otimes_\bC R)| \xrightarrow{\rK^\st(v)} | \colim_{\Delta^n_\tp \ra (\Spec R)^\an}\rK(A\otimes_\bC R)|,
$$

where the map is determined by the map of functors $\rK(B\otimes\cdot)\ra\rK(A\otimes\cdot)$.

Consider the cokernel of this map. Commuting the colimits we can express it as

$$
    \coker(\rK^\st(v)) \simeq |\colim_{\Delta^n_\tp \ra (\Spec R)^\an}(\coker(\rK(B\otimes_\bC R)\ra\rK(A\otimes_\bC R)))| .
$$

\ \ 

Now, since $R$ in the diagram is taking values in classical regular rings, the map $v\otimes R$ is a nilpotent extension, so, 
using functoriality of the cyclotomic trace map, we can apply Theorem~\ref{DGMTC} termwise: 

$$
    \coker(\rK^\st(v)) \simeq |\colim_{\Delta^n_\tp \ra (\Spec R)^\an}(\coker(\TC(B\otimes_\bC R)\ra\TC(A\otimes_\bC R)))| .
$$

For a connective 
$\bbQ$-algebra $S$, $\TC(S) \simeq \TC^-(S) \simeq \THH(S)^{hS^1} \simeq \HH(S / \bbQ)^{hS^1} \simeq \HC^-(S / \bbQ) $,  so all the finite coefficient type data vanishes 
and we get

$$
    \coker(\rK^\st(v)) \simeq |\colim_{\Delta^n_\tp \ra (\Spec R)^\an}(\coker(\HC^-(B\otimes_\bC R)\ra\HC^-(A\otimes_\bC R)))| .
$$

We can commute the colimits in the formula again, to get back the cokernel of realizations:

$$
    \coker(\rK^\st(v))\otimes\bbQ \simeq \coker(Re(HC^-(B\otimes_\bbC\cdot)) \ra Re(HC^-(A\otimes_\bbC\cdot))) .
$$

By Proposition~\ref{ReHC}, functoriality of $u$-inversion in $\HC^-$ and one more iteration of commuting the colimits, we get

$$
    \coker(\rK^\st(v))\otimes\bbQ \simeq \coker(Re(HP(B\otimes_\bbC\cdot)) \ra Re(HP(A\otimes_\bbC\cdot))) \simeq 
$$
    
$$
   \simeq Re(\coker(HP(B\otimes_\bbC\cdot) \ra HP(A\otimes_\bbC\cdot))) \simeq Re(0) \simeq 0 ,
$$

where for the second to the last equivalence we apply Theorem~\ref{GHP}.

\end{proof}

\ \ 

\begin{coro}\label{lcnil}

Let $v: B \ra A$ be a nilpotent extension of connective $\bbC$-dg-algebras. Then

$$
Ch^\tp_A\otimes \bbC: \rK^\tp(A) \otimes \bbC \ra \HP(A/\bbC)
$$ 

is an equivalence if and only if so is the map

$$
Ch^\tp_B\otimes \bbC: \rK^\tp(B) \otimes \bbC \ra \HP(B/\bbC) .
$$

\end{coro}

\begin{proof}

We have a natural commutative diagram

$$
\xymatrix{
\rK^\tp(B)\otimes\bbC \ar[r]^{\ \ \ Ch^\tp_B} \ar[d]& \HP(B/\bbC) \ar[d] \\
\rK^\tp(A)\otimes\bbC \ar[r]^{\ \ \ Ch^\tp_A}       & \HP(A/\bbC)      , \\
} 
$$

where the right column 
map is an equivalence by Theorem~\ref{GHP}. By Theorem~\ref{Kstnil}, the left column is an equivalence as well, since there is a natural equivalence $\rK^\tp(\cdot) \otimes \bbC \simeq \rK^\st(\cdot)\otimes_{ku}KU\otimes\bbC \simeq (\rK^\st(\cdot)\otimes\bbQ)\otimes_{ku_\bbQ}(KU_\bbQ)\otimes_\bbQ\bbC$ and this is a product of a nil-invariant and a constant functor. So the upper map is an equivalence iff so is the bottom.

\end{proof}

\ \ 

\section{Some examples and applications}


\subsection{Connective dg-algebras}

We remind that any smooth and proper $\bbC$-dg-category is Morita-equivalent to a smooth proper $\bbC$-dg-algebra, so the question about existence of a natural (Morita-invariant) rational structure on periodic cyclic homology of smooth proper $\bbC$-dg-categories essentially boils down to existence of a natural rational structure on periodic cyclic homology of dg-categories of the form $\Perf(B)$ for $B$ a smooth and proper $\bbC$-dg-algebra. As an application of Theorem~\ref{Kstnil}, we prove that $\rK^\tp$ provides such a structure in the case of connective proper $\bbC$-dg-algebras. The smoothness assumption is not necessary in this case.

\begin{theo}\label{lccnprop}

Let $B$ be a connective proper $\bbC$-dg-algebra. Then the complexified Chern character map

$$ 
    \rK^\tp(B)\otimes \bbC \ra \HP(B)
$$

is an equivalence.

\end{theo}

\begin{proof}

We use Corollary~\ref{lcnil} to reduce to the case of classical finite-dimensional algebra $B^\red$ without nilpotents, hence semi-simple. Now, by Wedderburn's theorem, we deduce that $B^\red$ is Morita-equivalent to a finite product of $\bbC$, so our statement follows from the corresponding equivalence for $\Spec\mathbb{C}$ and additivity.

\end{proof}

This theorem generalizes Proposition 4.37 \cite{Bla} because of the following simple observation.

\begin{prop}\label{Kstneg}

Let $v: B \ra A$ be a nilpotent extension of connective $\bbC$-dg-algebras and $A$ be a classical Noetherian $\bbC$-algebra of finite global dimension. Then the natural maps

$$
    \tilde \rK^\st (B) \lra \rK^\st (B) 
$$

and

$$
    \tilde \rK^\tp (B) \lra \rK^\tp (B) 
$$

are equivalences of spectra.

\end{prop}

This shows that in the setting of Proposition 4.37 \cite{Bla}, $\tilde\rK^\st(B) \simeq \rK^\st(B)$ and the pseudo-topological K-theory can safely be replaced by $\rK^\tp$ even without proving Theorem~\ref{lccnprop} (from which it would follow by dimension counting). 

\begin{remar}

There are simple examples when $\tilde\rK^\st(T) \ne \rK^\st(T)$, e. g. taking $T:=\bbC[t_0,t_1]/(t_0t_1(1-t_0-t_1))$ works: $\tilde\rK^\st(\bbC[t_0,t_1]/(t_0t_1(1-t_0-t_1))) \simeq ku\oplus \Sigma ku$, while $\rK^\st(\bbC[t_0,t_1]/(t_0t_1(1-t_0-t_1))) \simeq \Omega ku\oplus ku$, so $\rK^\st_{-1}(\bbC[t_0,t_1]/(t_0t_1(1-t_0-t_1)))\simeq \bbZ$. However, it seems a tricky task to find a decent dg-category $T$, such that $\tilde\rK^\tp(T) \ne \rK^\tp(T)$. 

\end{remar}

\ \ 

\begin{proof}[Proof of Proposition~\ref{Kstneg}]

For every regular classical commutative $\bbC$-algebra $D$, if $B \ra A$ is a nil-extension, $B\otimes D \ra A\otimes D$ is a nil-extension as well. Also Noetherian $\bbC$-algebras of finite global dimension have no negative K-theory, so $\tilde\rK(A\otimes D) \simeq \rK(A\otimes D)$ under these assumptions. Using nilinvariance of non-positive algebraic K-theory groups Proposition~\ref{Kneg} and Corollary~\ref{Kstformula}, we get equivalences:

$$
    \rK^\st(B) \simeq |[n] \mapsto \colim_{\Delta^n_\tp \ra D^\an}\rK(B\otimes_\bC D)| \simeq |[n] \mapsto \colim_{\Delta^n_\tp \ra D^\an}\tilde\rK(B\otimes_\bC D)| \simeq \tilde\rK^\st(B) 
$$

and the equality 

$$
    \rK^\tp(B) \simeq \tilde\rK^\tp(B) 
$$

follows. 

\end{proof}

Using Proposition~\ref{lccnprop} (and nil-invariance of periodic cyclic homology, in order to get a simpler answer), one can also write formulas expressing HP in terms of the realization of the stack of perfect (or projective) modules, analogous to \cite{Bla}, section 4.7 in this context.

\ \ 

We also want to mention that Proposition 4.37 \cite{Bla} can be generalized in another way.

\begin{prop}\label{lcfddg}

Let $B$ be a smooth $\bbC$-dg-algebra, which is finite dimensional, i.e.  $B$ is Morita-equivalent to a $\bbC$-dg-algebra $A$, such that after forgetting the multiplication, the differential and the grading on $A$ one gets a finite-dimensional $\bbC$-vector space: $\mathrm{dim}(\bigoplus_{i\in \bbZ}A_i) < \infty$. Then the complexified topological Chern character $\Ch^\tp: \rK^\tp(B)\otimes \bbC \ra \HP(B /\bbC)$ is an equivalence.

\end{prop}

Note that the property of being finite-dimensional is stronger than properness: dg-algebras are rarely formal (or have a finite-dimensional model). In particular, by \cite{Orl}, Corollary 2.21, $\rK_0(T) \simeq \bbZ^{\oplus n}$ for $T$ Morita-equivalent to a smooth finite dimensional $B$, but even for a proper smooth variety $X$ $\rK_0(X)$ is usually not finitely generated.

\begin{proof}

In a recent paper, Orlov generalizes the Auslander construction and shows that, for a finite dimensional $\bbC$-dg-algebra $A$, there is a fully faithful embedding of $\Perf(A)$ into the category of perfect complexes over a dg-algebra $\cE$ constructed from $A$ such that the category $\Perf(\cE)$ comes with a full exceptional collection (\cite{Orl}, Theorem 2.19). 

The equivalence $\rK^\tp(\cE)\otimes \bbC \simeq \HP(\cE /\bbC)$ follows from existence of a full exceptional collection on $\cE$ and the case of $\Spec\bbC$, since both invariants are weakly localizing. Since we assumed that $A/\bbC$ is smooth, the embedding $\Perf(A) \lra \Perf(\cE)$ is admissible, hence, by functoriality, the equivalence $\rK^\tp(A)\otimes \bbC \simeq \HP(A /\bbC)$ follows from the equivalence $\rK^\tp(\cE)\otimes \bbC \simeq \HP(\cE /\bbC)$, since the former is a retract in the latter.

\end{proof}

\ \ 

\subsection{Derived schemes}

We can use our nilinvariance result to generalize \cite{Bla}, Proposition 4.32 to the derived setting. Restricted to the class of commutative dg-algebras, Theorem~\ref{Kstnil} states that semi-topological K-theory of a derived affine scheme $\fX/\bbC$ only depends on the corresponding reduced classical subscheme $X/\bbC$. This observation stays true globally.

\begin{prop}\label{Kstnilsch}

Let $\fX$ be a quasi-compact quasi-separated derived scheme over $\bbC$. 
Denote $X := (\pi_0\fX)^{\text{red}} \ra \fX$ the reduction of the classical truncation of $\fX$. Let $F$ be a nil-invariant weakly localizing invariant. Then the corresponding morphism $F(\fX) \ra F(X)$ is an equivalence. 

\end{prop}

Applying this proposition to $F:= \rK^\st$ and $F: = \HP(\cdot /\bbC)$ together with Proposition~\ref{lcsch}, we get the following corollary.

\begin{coro}\label{lcdsch}

Let $\fX/\bbC$ be a derived scheme, such that its classical part $\pi_0\fX$ is a separated scheme of finite type (this includes laft schemes of \cite{GR}).
Then there are canonical equivalences: 

$$
\rK^\tp(\Perf(\fX)) \simeq \rK^\tp(\Perf(\pi_0\fX)) \simeq KU((\pi_0\fX)^\han),
$$

$$
\rK^\tp(\Perf(\fX))\otimes \bbC \simeq \HP(\Perf(\fX)) \simeq \HP(\Perf(\pi_0\fX)).  
$$

\end{coro}

\ \ 

As the proof of Proposition~\ref{Kstnilsch} below suggests, this equivalences can be understood by covering $\fX$ by affines and proceeding inductively. The reason behind the finiteness assumption is that, unlike (semi-)topological K-theory, periodic cyclic homology does not commute with filtered colimits.

\ \ 

\begin{proof}[Proof of Proposition~\ref{Kstnilsch}]

The basic idea is to prove that weakly localizing invariants satisfy `derived Zariski descent', by mimicking the standard argument from \cite{TT}. 
In the qcqs case, this allows us to reduce the question to derived affine case, which is 
secured by Theorem~\ref{Kstnil} (and in the case of $\pi_0\fX$ separated of finite type allows to reduce to derived affine schemes with finite type classical part where one has comparison between two versions of topological K-theory and periodic cyclic homology).

So the statement follows from the following observation (basically, contained in \cite{CMNN}, Appendix A).

\begin{prop}\label{derZar}

Let $\fX$ be a qcqs derived scheme together with two open quasicompact subschemes $\fU_1\subset \fX \supset \fU_2$ and let $F$ be a weakly localizing invariant. Then the triple

$$
    F(\fX) \ra F(\fU_1)\oplus F(\fU_2) \ra F(\fU_1\times_{\fX}\fU_2)
$$

is a fiber sequence.

\end{prop}

\begin{proof}[Proof of Proposition~\ref{derZar}]

The statement follows from two facts. One is the existence of an exact triple of triangulated dg-categories

$$
    \Perf_Z(\fX) \lra \Perf(\fX) \lra \Perf(\fU_1)  
$$

for $j_1: \fU_1 \subset \fX$ as before, $Z$ a closed complement in the underlying topological space of $\fX$ to $\fU_1$ and $\Perf_Z(\fX)$ the sub-dg-category of perfect complexes supported on $Z$, i.e., such $\cF\in \Perf(\fX)$ that $j^*\cF = 0$ (\cite{CMNN}, Corollary A.10).
The other observation is the independence of $\Perf_Z$ of the scheme $Z$ is contained in: $j_2^*: \Perf_Z(\fX) \ra \Perf_Z(\fU_2)$ is an equivalence (\cite{CMNN}, Proposition A.14). This allows us to identify the fibers of $F(\fX) \ra F(\fU_1)$ and $F(\fU_2) \ra F(\fU_1\times_{\fX}\fU_2)$, finishing the proof.

\end{proof}

\end{proof}

\ \ 

Corollary~\ref{lcdsch} can be formulated slightly more generally.

\begin{coro}

Let $\fY/\bbC$ be a derived algebraic space, together with a surjective etale map $\fX \ra \fY$, such that $\fX$ is as in Corollary~\ref{lcdsch}. 
Then the topological Chern character

$$
\rK^\tp(\Perf(\fY))\otimes \bbC \ra \HP(\Perf(\fY)).  
$$

is an equivalence.

\end{coro}

\begin{proof}

By Theorem A.4 \cite{CMNN}, the topological Chern character, being a map of weakly localizing invariants valued in $\bbQ$-modules, is a map of etale sheaves, so we can compute both invariants by Cech diagrams, which by Corollary~\ref{lcdsch} are equivalent. 

\end{proof}

This minor generalization comes for free, but we also are optimistic about lattice conjecture for stacky quotients of derived schemes.

\ \

\subsection{Group algebras}

Let $G$ be a group. We are interested in understanding the topological Chern character map for the corresponding group algebra $\bbC[G]$, at least after complexification. In the next subsection, this will be applied to prove the lattice conjecture for $\infty$-categories of local systems on reasonable topological spaces.

\ \ 

We start by considering two simple classes of groups for which all the necessary work has already been done in previous sections.

\begin{prop}\label{lcfabgp}

Let $G$ be a finite group or a finitely generated abelian group. Then the topological Chern character induces an equivalence 

\[
\rK^\tp(\bbC[G])\otimes\bbC\simeq \HP(\bbC[G]).
\]

\end{prop}

\begin{proof}

In the finite group case the statement follows from Theorem~\ref{lccnprop}; in the finitely generated abelian group case the equivalence follows from Proposition~\ref{lcsch} .

\end{proof}

\begin{remar}

Of course, in these cases one can always write down an explicit formula for semi-topological K-theory: the semi-topological K-theory of $\bbC[\bbZ^{\oplus n}]$ is just the connective part of the topological K-theory of $n$-torus, adding cyclic groups corresponds to taking several copies of it; in the case of a finite group, one can reduce to the semi-topological K-theory of several copies of $\bbC$: $\rK^\st(\bbC[G])\simeq ku^{\oplus m}$.

\end{remar}

\ \ 

For a general group, one cannot expect such a result to be true. 
Nevertheless, the lattice conjecture can be proved for a surprisingly large class of groups via reducing to their finite subgroups.

\ \ 

In \cite{Bur}, it was shown that the target of the map

$$
    \rK^\tp(\bbC[G])\otimes \bbC \ra \HP(\bbC[G])
$$
can always be understood via homology of the cyclic subgroups of $G$, namely the isomorphism 

\begin{equation} \label{HPcalc}
\xymatrix{
\HP_\ast(\bbC [G]) \simeq \bigoplus_{[g] \in \langle G \rangle^{fin}} H_{[\ast]}(C_g;\bbC) \oplus \bigoplus_{[g] \in \langle G \rangle^{\infty}} T^G_\ast(g;\bbC),
}
\end{equation}
was proved, where $\langle G \rangle^{fin}$ / $\langle G \rangle^{\infty}$ stand for the conjugacy classes of the finite / infinite order elements in $G$,  $C_g$ is the centralizer of $g\in G$, $H_{[*]}(C_g;\bbC) := \prod_{i \in \bbZ} H_{*+2i}(C_g;\bbC)$, $T^G_*(g;\bbC):= \varprojlim \big( \cdots \xrightarrow{S} H_{\ast+2n}(C_g/\langle g\rangle;\bbC) \xrightarrow{S} H_{\ast+2n-2}(C_g/\langle g\rangle;\bbC) \xrightarrow{S} \cdots \big)$ and $S$ are Gysin maps.

As we remind below, the image of the topological Chern character projects as zero on the second term in the formula above, hence vanishing of this term is a necessary condition for  the lattice conjecture
to be true for $\bbC[G]$. The condition of such vanishing for a group $G$ is sometimes called (generalized) Burghelea conjecture. The \cite{EM} paper is a good source on the current state of the Burghelea conjecture, in particular it has been shown there that the conjecture holds for 
a large class of groups.
 
\begin{remar}
 
In Section 5, \cite{EM}, there were also given counterexamples to the Burghelea conjecture with relatively good finiteness properties. Since the topological Chern character map knows nothing about $T^G$ term, those group algebras provide (very non-proper) counterexamples to the lattice conjecture. Note also that Burghelea did not in any way expect vanishing of $T^G$-term  for all groups: he only asked if it is true for the groups admitting a finite classifying space and he also gave a (not finitely generated) example of a group where the $T^G$-term is nonzero.
 
\end{remar}

\ \ 

The left hand side  of the map 

$$
   \rK^\tp(\bbC[G])\otimes \bbC \ra \HP(\bbC[G])
$$
is built via colimits from the spectra  of the form $\rK(R[G])\otimes \bbC$ where $R$ runs through the category of finitely generated regular $\bbC$-algebras. 
K-theory of group algebras has been studied a lot for the last decades in relation to the Farrell-Jones's conjecture. Rationally for $\bbC$-algebras this conjecture becomes similar to the Burghelea's computation: roughly speaking it states that the K-theory of $\bbC[G]$ can be computed in terms of the K-theories of $\bbC[H]$ for finite cyclic subgroups $H\subset G$.

In \cite{LR06}, the authors show that an additive invariant $I$ can be promoted to an $\Or G$-specrum $\bI = \bI_R$ (i.e. a functor from the subcategory of $G$-sets of the form $G/H$ to spectra), such that $\bI_R(G/H) \simeq I(R[H])$. Then one can left Kan extend an $\Or G$-spectrum $\bI$ to a functor $(-)\wedge_{\Or G}\bI \in \Fun(\mathrm{Top}^G, \Sp)$ -- homotopy groups of its values are usually denoted by $\rH^G_*(-,\bI)$. Restricting to $R=\bbC$, applying this construction to the Chern character and substituting the map of $G$-spaces $E_\cF G\ra G/G = pt$, one obtains a commutative square:

$$
\xymatrix{
E_\cF G\wedge_{\Or G}\bK(\bbC[-]) \ar[r] \ar[d]& G/G\wedge_{\Or G}\bK(\bbC[-]) \simeq \rK(\bbC[G]) \ar[d] \\
E_\cF G\wedge_{\Or G}\bHP(\bbC[-]) \ar[r]       & G/G\wedge_{\Or G}\bHP(\bbC[-])\simeq \HP(\bbC[G])      . \\
} 
$$

Here $E_\cF G\in \mathrm{Top}^G$ is the classifying space for a family $\cF$ of subgroups in $G$, which can be characterized by requiring its subspace of $H$-invariants $(E_\cF G)^H$ to be contractible if $H\in \cF$, and empty otherwise. It follows that the usual colimit diagram for the left Kan extension $E_\cF G\wedge_{\Or G}\bI$ does not have $\bI(G/H) \simeq I(R[H])$-terms for $H\notin \cF$, and for $\cF = \langle(e)\rangle$ $E_\cF G\wedge_{\Or G}\bI\simeq \rB G_+\wedge I(R)$.

We refer to \cite{LR06} for details and to \cite{RV}, \cite{LR05}, \cite{Lueck}
for more information on the Farrell-Jones conjecture and relevant equivariant homotopy theory; in the case of a finite type regular $\bbC$-algebra $R$ and a group $G$, the rational Farrell-Jones conjecture states that for a family of finite cyclic groups $\cF = FCyc_G=:FCyc$, the assembly map $E_\cF G\wedge_{\Or G}\bK(R[-])_\bbQ 
\ra G/G\wedge_{\Or G}\bK(R[-])_\bbQ \simeq \rK(R[G])_\bbQ$ is an equivalence (here Propositions 2.14 and 2.20 of \cite{LR05} are used to reduce to $FCyc$ instead of $VCyc$). As before, we use that the realization, Bott-inverting and complexification functors commute with colimits to get a commutative square:

\begin{equation}\label{FJKHP}
\xymatrix{
E_{FCyc}(G)\wedge_{\Or G}\bK^\tp(\bbC[-]) \ar[r] \ar[d]& G/G\wedge_{\Or G}\bK^\tp(\bbC[-]) \simeq \rK^\tp(\bbC[G]) \ar[d] \\
E_{FCyc}(G)\wedge_{\Or G}\bHP(\bbC[-]) \ar[r]       & G/G\wedge_{\Or G}\bHP(\bbC[-])\simeq \HP(\bbC[G])      ; \\
}     
\end{equation}
whenever the Farrel-Jones conjecture is true for $G$, the top map is an equivalence of $\bbC[\beta]$-modules. The bottom map is another way to write the embedding of the first term in (\ref{HPcalc}), in particular, it is an equivalence iff the Burghelea conjecture is satisfied by $G$. Now we formulate the main result of this subsection.

\begin{prop}\label{lcgpalg}

The complexified topological Chern character is an equivalence for $\bbC[G]$ if the group $G$ belongs to one of the following classes:

(i) hyperbolic groups;

(ii) finite-dimensional CAT(0) groups; 

(iii) mapping class groups of compact orientable surfaces with punctures;

(iv) systolic groups;

(v) compact 3-manifold groups; 

(vi) Coxeter groups;

(vii) right-angled Artin groups;

(viii) solvable groups of finite Hirsch lenghth.


\end{prop}

\begin{proof}

The Farrell-Jones is true for these classes of groups (\cite{BLR}, \cite{BL}, \cite{BB}, \cite{Rou}, \cite{DJ}, \cite{W}), so the top map is an equivalence. The Burghelea conjecture is true for all these groups (\cite{Ji}, \cite{EM}), so the bottom map is an equivalence. The left map is an equivalence as well, since the complexified topological Chern character becomes an equivalence of functors after restriction to the subcategory of algebras of the form $\bbC[H]$ for $H$ any finite (cyclic) group (Proposition~\ref{lcfabgp}). Therefore the right vertical map is an equivalence.

\end{proof}

\begin{remar}

When the group $G$ has no non-trivial finite cyclic subgroups, i.e. $G$ is torsion-free, the square (\ref{FJKHP}) takes the simpler form

\begin{equation}
\xymatrix{
\rB G_+\wedge\rK^\tp(\bbC) \ar[r] \ar[d]& \rK^\tp(\bbC[G]) \ar[d] \\
\rB G_+\wedge\HP(\bbC) \ar[r]       &  \HP(\bbC[G])      ; \\
}     
\end{equation}
equivariant theory is no longer necessary here.

\end{remar}

\begin{remar}\label{FJspecul}

Let us say that the group $G$ satisfies the semi-topological / topological / rational topological Farrel-Jones conjecture if the assembly map $E_{VCyc}(G)\wedge_{\Or G}\bI\ra I(\bbC[G])$ is an equivalence for the corresponding invariant. Of course, the usual Farrel-Jones conjecture for regular $\bbC$-algebras $\implies FJ^\st \implies FJ^\tp \implies FJ^\tp_\bbQ$. By the proof of Proposition~\ref{lcgpalg}, in the triple $\langle$ the Burghelea conjecture for $G$, the lattice conjecture for $\bbC[G]$, $FJ^\tp_\bbQ$ $\rangle$, any two statements imply the third. This might be helpful, for instance, for constructing a counterexample to the Farrell-Jones conjecture: it is enough to provide a counterexample to the lattice conjecture of a form $\bbC[G]$ where $G$ is a group for which Burghelea's conjecture is secured. Actually, instead of $\bbC[G]$ one can consider the algebra $C_\ast(\Omega M)$ for any $M$, s.t. for $\pi_1M$ Burghelea conjecture is true, see Corollary~\ref{locsysnil} and Remark~\ref{FJlocsys}.

\end{remar}




\ \ 

\subsection{Local systems}

We finish with one more application of Theorem~\ref{Kstnil}. Nilinvariance of semi-topological K-theory implies that, applied to the $\infty$-category of local systems of chain complexes on a decent space $M$, semi-topological K-theory depends only on the fundamental group $\pi_1M$. 

\ \ 

Let $M$ be a pointed connected topological space which is locally of singular shape in the sense of \cite{LurHA}, Definition A.4.15. Then there is a fully faithful functor $\Spc_{\Sing(M)}\lra \Shv(M,\Spc)$ from Kan fibrations over $\Sing(M)$ to the category of sheaves on $M$ with the essential image the subcategory of locally constant sheaves 
(\cite{LurHA}, Theorem A.4.19, see also \cite{Toe}, Theoreme 2.13), so, roughly speaking, the theory of locally constant sheaves on nice enough spaces can be defined on the level of homotopy types.

Now we define the $\infty$-category of local systems of $\bbC$-complexes on a homotopy type. We consider the functor

$$
    \LOC^\ast: \Spc \lra \DGCAT_\bbC^\op
$$
to the category of presentable dg-categories given by $E \mapsto \Fun(E, Mod_\bbC)$ (cf. \cite{LurKM}, Lecture 21), which can also be defined as the left Kan extension from its value on $pt$. The image of the map $f: E_1 \ra E_2$ is given by the precomposition functor $f^\ast$, which has both left and right adjoints $f_!$ and $f_\ast$. The $!$-pushforwards preserve compact objects, so we can define a functor 

$$
\Loc(-,\bbC): \Spc \lra dgCat_\bbC
$$
$$
    E \mapsto (\Fun(E, Mod_\bbC))^c; \ \ f \mapsto f_! .
$$
The functor $\Loc(-,\bbC)$ is equivalent to $E \mapsto \Perf(C_\ast(\Omega E))$. This seems to be a folklore statement, which is partially proved in \cite{LurKM}, Lecture 21; we omit the details here, since in the end we are only interested in the categories of the form $\Perf(C_\ast(\Omega E))$, which can be taken as the definition of $\Loc(E,\bbC)$. By precomposing with $M \mapsto \Sing_\bullet(M)$, \ $\Loc(-,\bbC)$ can be made into a functor from the 1-category of topological spaces, which we, by abuse of notation, will again denote $\Loc(-,\bbC)$.

\ \ 

Applying Theorem~\ref{Kstnil}, we get the following corollary.

\begin{coro}\label{locsysnil}

Let $M$ be a pointed connected locally contractible topological space. Then there is a natural equivalence $\rK^\st(\Loc(M,\bbC)) \simeq \rK^\st(\bbC[\pi_1M])$.

\end{coro}

In particular, when the fundamental group of $M$ is nice enough, we deduce the lattice conjecture for the category of local systems.

\begin{theo}\label{lclocsys}

Let $M$ be a pointed connected locally contractible topological space such that its fundamental group belongs to one of the following classes:

(i) hyperbolic groups;

(ii) finite-dimensional CAT(0) groups; 

(iii) mapping class groups of compact orientable surfaces with punctures;

(iv) systolic groups;

(v) compact 3-manifold groups; 

(vi) Coxeter groups;

(vii) right-angled Artin groups;

(viii) solvable groups of finite Hirsch lenghth.


Then the topological Chern character

$$
    \rK^\tp(\Loc(M,\bbC))\otimes\bbC \ra \HP(\Loc(M,\bbC))
$$

is an equivalence.

\end{theo}

\begin{proof}

Apply Corollary~\ref{lcnil} and Proposition~\ref{lcgpalg}.

\end{proof}

\begin{remar}\label{FJlocsys}

The Corollary~\ref{locsysnil} shows that extending the lattice conjecture from group algebras to local systems does not bring new objects, but adds flexibility. Larger supply of maps in the category of spaces should make semi-topological K-theory of local systems more computable. Remark~\ref{FJspecul} suggests to search for a space $M$ with a fundamental group $G$, satisfying the Burghelea conjecture, such that the complexified topological Chern character has a nontrivial kernel -- the fundamental group of such a space would provide a counterexample to the usual Farrell-Jones conjecture. It might be easier sometimes to find a model with nice geometrical properties for such $M$ than for $\rB G$.

\end{remar}

\begin{remar}

The category $\Perf(C_\ast(\Omega M))$ has one more incarnation. When $M$ is a smooth connected manifold, $\Perf(C_\ast(\Omega M))$ is Morita-equivalent to the wrapped Fukaya category of the cotangent bundle $T^\ast M$ (see Theorem 1.1, Corollary 6.1 \cite{GPS}, or Theorem 1.1 \cite{Ab} for the earlier treatment of closed oriented case). Theorem~\ref{lclocsys} guarantees that, under certain conditions on $\pi_1M$, $\HP(\cW(T^\ast M))$ comes with an integral structure. Since the topological Chern character is a map of weakly localizing invariants, Lemma 6.2 and Corollary 6.3 \cite{GPS} allow to extend this structure to certain plumbings -- it is enough to check  joint fully-faithfullness of the maps in the diagrams (6.2), (6.3). The next logical step would be to try and generalize Theorem~\ref{lclocsys} to certain categories of constructible sheaves, which are related to more complicated wrapped Fukaya categories.

\end{remar}




\ \


\begin{thebibliography}{}

\bibitem[Ab]{Ab} \rm M. Abouzaid, 
\it A cotangent fibre generates the Fukaya category\rm , Adv. Math. 228 (2011), no. 2,
894–939.

\bibitem[AH]{AH} \rm B. Antieau, J. Heller, \it Some remarks on topological K-theory of dg categories\rm , arXiv preprint 1709.01587.

\bibitem[AV]{AV} \rm B. Antieau, G. Vezzosi, \it A remark on the Hochschild-Kostant-Rosenberg theorem in characteristic p\rm , arXiv preprint 1710.06039.

\bibitem[BB]{BB} \rm A. Bartels, M. Bestvina, \it The Farrell-Jones conjecture for mapping class groups\rm . Preprint, available at arXiv:1606.02844, 2016.


\bibitem[Bla]{Bla} \rm A. Blanc, \it Topological K-theory of complex
noncommutative spaces\rm , Compositio Math. 152 (2016), 489–555.


\bibitem[BL]{BL} \rm A. Bartels, W. Lück, \it The Borel Conjecture for hyperbolic and CAT(0)-groups\rm . Ann. of Math. (2), 175(2):631–689, 2012.

\bibitem[BLR]{BLR} \rm A. Bartels, W. Lück, H. Reich, \it The K-theoretic Farrell-Jones Conjecture for hyperbolic groups\rm . Invent. Math., 172(1):29–70, 2008.

\bibitem[Bur]{Bur} \rm D. Burghelea, \it The cyclic homology of the group rings\rm . Commentarii Mathematici Helvetici. 60. 354-365. 10.1007/BF02567420 (1985). 


\bibitem[CMNN]{CMNN} \rm D. Clausen, A. Mathew, N. Naumann, J. Noel, \it Descent in algebraic K-theory and a conjecture of Ausoni-Rognes\rm , arXiv preprint 1606.03328 (2017).

\bibitem[Cohn]{Cohn} \rm L. Cohn, \it Diﬀerential graded categories are k-linear stable inﬁnity categories\rm , arXiv: 1308.2587, 2013.

\bibitem[DGM]{DGM} \rm Bjørn Ian Dundas, Thomas Goodwillie, and Randy McCarthy. \it The local structure of algebraic K-theory\rm , volume 18. Springer Science \& \ Business Media, 2012.

\bibitem[DJ]{DJ} \rm M. W. Davis, T. Januszkiewicz, \it Right-angled Artin groups are commensurable with right-angled Coxeter groups\rm . J. Pure Appl. Algebra, 153(3):229–235, 2000.

\bibitem[EM]{EM} \rm A. Engel, M. Marcinkowski. \it Burghelea conjecture and asymptotic dimension of groups\rm , 
Journal of Topology and Analysis 12 (02), 321-356

\bibitem[FW01]{FW01} \rm E. Friedlander, M. Walker. \it Comparing K-theories for complex varieties\rm .
Amer. J. Math., 123(5):779–810, 2001.

\bibitem[FW03]{FW03} \rm E. Friedlander, M. Walker. \it Rational isomorphisms between K-theories and cohomology theories\rm , Inventiones mathematicae 154 (2003), no. 1, 1–61.

\bibitem[FW05]{FW05} \rm E. Friedlander, M. Walker, \it Semi-topological K-theory\rm , Handbook of K-theory (2005), 877–924.

\bibitem[G85]{G85} \rm T.G. Goodwillie, \it Cyclic homology, derivations, and the free loopspace\rm .
Topology 24 (1985), no. 2, 187215.


\bibitem[GPS]{GPS} \rm S. Ganatra, J. Pardon, V. Shende, \it Microlocal Morse theory of wrapped Fukaya categories\rm , arXiv:1809.08807, 2020.

\bibitem[GR]{GR} \rm D. Gaitsgory, N. Rozenblyum, \it A Study in Derived Algebraic Geometry Vol. I. Correspondences and duality\rm , volume 221 of Mathematical Surveys and Monographs. American Mathematical
Society, Providence, RI, 2017.

pp. 1–26.


\bibitem[Ji]{Ji} \rm R. Ji, \it Nilpotency of Connes’ Periodicity Operator and the Idempotent Conjectures\rm , K-Theory 9 (1995), 59–76.

\bibitem[Kal]{Kal} \rm D. Kaledin, \it Spectral sequences for cyclic homology\rm . In Algebra, geometry, and physics in the 21st century,
volume 324 of Progr. Math., pages 99–129. Birkhauser/Springer, Cham, 2017.

\bibitem[KKP]{KKP} \rm L. Katzarkov, M. Kontsevich, and T. Pantev, \it Hodge theoretic aspects of mirror symmetry\rm ,
arXiv preprint arxiv:0806.0107 (2008).

\bibitem[Lod]{Lod} \rm J.-L. Loday, \it Cyclic homology\rm . Grundlehren der Mathematischen Wissenschaften
301. Springer-Verlag, Berlin, 1998.

\bibitem[LR05]{LR05} \rm W. Lück, H. Reich. \it The Baum-Connes and the Farrell-Jones Conjectures
in K- and L-theory\rm . In Handbook of K-theory. Vol. 2, pages 703–842. Springer, Berlin,
2005.

\bibitem[LR06]{LR06} \rm W. Lück and H. Reich. \it Detecting K-theory by cyclic homology\rm . Proc. London Math.
Soc. (3), 93(3):593–634, 2006.

\bibitem[Lück]{Lueck} \rm Wolfgang Lück. \it Isomorphism Conjectures in K- and L-Theory\rm . In preparation, preliminary version available at him.uni-bonn.de/lueck/.


\bibitem[LurHTT]{LurHTT} J. Lurie, {\em Higher topos theory}

\bibitem[LurHA]{LurHA} J. Lurie, {\em Higher algebra}, 2017.

\bibitem[LurKM]{LurKM} \rm J. Lurie, \it Lecture 20, 21, Algebraic K-Theory and Manifold Topology (Math 281)\rm , lecture notes, https://www.math.ias.edu/~lurie/281notes/Lecture20-Lower.pdf

\bibitem[Orl]{Orl} \rm D. Orlov, \it Finite-dimensional differential graded algebras and their geometric realizations\rm , arXiv preprint arxiv:1907.08162 (2019)

\bibitem[Ras]{Ras} \rm S. Raskin, \it On the Dundas-Goodwillie-McCarthy theorem\rm , arXiv preprint arxiv:1807.06709 (2018)


\bibitem[Rou]{Rou} \rm S. K. Roushon, \it The Farrell-Jones isomorphism conjecture for 3-manifold groups\rm . J. K-Theory, 1(1):49–82, 2008.

\bibitem[RV]{RV} \rm H. Reich, M. Varisco. \it Algebraic K-theory, assembly maps, controlled algebra, and
trace methods\rm . In Space—time—matter, pages 1–50. De Gruyter, Berlin, 2018.

\bibitem[SP]{SP} \rm The Stacks Project Authors. \it Stacks Project\rm . http://stacks.math.columbia.edu, 2020.

\bibitem[Tab]{Tab} \rm  G. Tabuada, \it Invariants additifs de dg-catgories\rm . Internat. Math. Res. Notices 53 (2005),
33093339.

\bibitem[Toën]{Toe} \rm B. Toën, \it Vers une interpr
etation Galoisienne de la th
eorie de l’homotopie\rm , Cahiers de topologie et geometrie differentielle categoriques, Volume XLIII (2002), 257-312.

\bibitem[TT]{TT} \rm  R. W. Thomason and Thomas Trobaugh. \it Higher algebraic K-theory of schemes and of derived categories\rm .
In The Grothendieck Festschrift, Vol. III, volume 88 of Progr. Math., pages 247–435. Birkh
auser Boston,
Boston, MA, 1990


\bibitem[W]{W}
Christian Wegner. The Farrell-Jones conjecture for virtually solvable groups. J. Topol.,
8(4):975–1016, 2015.



\end{thebibliography}
\end{document}